\documentclass[12pt]{article}
\usepackage{amsmath,amssymb,amsthm}
\usepackage{cases}
\usepackage{amsfonts}
\usepackage{color,xcolor,cite}
\usepackage[integrals]{wasysym}
\usepackage[left=2.0cm,right=2.0cm,top=2.0cm,bottom=2.0cm]{geometry}
\usepackage[colorlinks,citecolor=blue,urlcolor=blue]{hyperref}

\numberwithin{equation}{section}

\newtheorem{theorem}{Theorem}[section]

\newtheorem{lemma}{Lemma}[section]
\newtheorem{proposition}{Proposition}[section]

\newtheorem{definition}{Definition}[section]
\newtheorem{remark}{Remark}[section]

\newcommand{\bal}{\begin{align}}
\newcommand{\bbal}{\begin{align*}}
\newcommand{\beq}{\begin{equation}}
\newcommand{\eeq}{\end{equation}}
\newcommand{\bca}{\begin{cases}}
\newcommand{\eca}{\end{cases}}
\def\div{\mathord{{\rm div}}}
\newcommand{\pa}{\partial}
\newcommand{\fr}{\frac}
\newcommand{\na}{\nabla}
\newcommand{\De}{\Delta}

\newcommand{\cd}{\cdot}

\newcommand{\w}{\widetilde{w}}
\newcommand{\su}{\widetilde{u}}
\newcommand{\R}{\mathbb{R}}

\newcommand{\ro}{\rho}
\newcommand{\T}{\mathcal{T}}
\newcommand{\sL}{W_{2,p}}
\newcommand{\D}{\Delta}

\newcommand{\g}{\big}

\begin{document}
\title{On the continuity of the solution map of the Euler-Poincaré equations in Besov spaces}

\author{Min Li$^{1}$ \footnote{E-mail: limin@jxufe.edu.cn} \\
\small $^1$ Department of Mathematics, Jiangxi University of Finance and Economics, Nanchang, 330032, China
}

\date{}

\maketitle\noindent{\hrulefill}

{\bf Abstract:} By constructing a series of perturbation functions through localization in the Fourier domain and using a symmetric form of the system, we show that the data-to-solution map for the Euler-Poincaré equations is nowhere uniformly continuous in $B^s_{p,r}(\R^d)$ with $s>\max\{1+\frac d2,\frac32\}$ and $(p,r)\in (1,\infty)\times [1,\infty)$. This improves our previous result which shows the data-to-solution map for the Euler-Poincaré equations is non-uniformly continuous on a bounded subset of $B^s_{p,r}(\R^d)$ near the origin.

{\bf Keywords:} Euler-Poincaré equations;  Nowhere uniformly continuous; Besov spaces; Data-to-solution map.

{\bf MSC (2010):} 35Q35, 35Q51, 35L30
\vskip0mm\noindent{\hrulefill}

\section{Introduction}\label{sec1}
In this paper, we consider the Cauchy problem in $\R^d$ for Euler-Poincaré equations
\begin{equation}\label{E-P}
\begin{cases}
\partial_tm+u\cdot \nabla m+\nabla u^T\cd m+(\mathrm{div} u)m=0, \qquad &(t,x)\in \R^+\times \R^d,\\
m=(1-\De)u,\qquad &(t,x)\in \R^+\times \R^d,\\
u(0,x)=u_0,\qquad &x\in \R^d.
\end{cases}
\end{equation}
The equations \eqref{E-P} were first introduced by Holm, Marsden, and Ratiu in \cite{hmr1,hmr2} as a high dimensional generalization of the following Camassa-Holm equation for modeling and analyzing the nonlinear shallow water waves :
\begin{align}\label{che}\tag{CH}
  m_{t}+um_{x}+2u_{x}m=0, \ m=u-u_{xx}.
  \end{align}
Indeed, when $d=1$ the Euler-Poincar\'{e} equations are the same as the Camassa-Holm equation (\ref{che}). Also, the Euler-Poincaré equations were investigated as the system describe geodesic motion on the diffeomorphism group with respect to the kinetic energy norm in \cite{hs}. 

 For $d=1$, the equation (\ref{che}) was introduced by Camassa and Holm\cite{ch} as a bi-Hamiltonian model for shallow water waves. Most importantly, CH equation has peakon solutions of the form $Ce^{-|x-Ct|}$ which aroused a lot of interest in physics, see \cite{c5,t}. There is an extensive literature about the strong well-posedness, weak solutions and analytic or geometric properties of the CH equation, here we name some. Local well-posedness and ill-posedness for the Cauchy problem of the CH equation were investigated in \cite{ce2,d2,glmy}. Blow-up phenomena and global existence of strong solutions were discussed in \cite{c2,ce2,ce3,ce4}. The existence of global weak solutions and dissipative solutions were investigated in \cite{bc1,bc2,xz1}, more results can be found in the references therein.

The first rigorous analysis of the Euler-Poincaré equations \eqref{E-P}  was done by Chae and Liu \cite{cli}, they eatablished the local existence of weak solution in $W^{2,p}(\R^d),\ p>d$ and local existence of unique classical solutions in $H^{s}(\R^d),\ s>\frac d 2+3$. Yan and Yin \cite{yy} further discussed the local existence and uniqueness of the solution to \eqref{E-P} in Besov spaces. On the other hand, Li, Yu and Zhai \cite{lyzz} proved that the solutions to \eqref{E-P} with a large class of smooth initial data blows up in finite time or exists globally in time, which settled an open problem raised by Chae and Liu \cite{cli}. Later, Luo and Yin have obtained a new blow-up result in the periodic case by using the rotational invariant properties of the equation\cite{luoy}. For more results of Euler-Poincar\'{e} equations, see \cite{luoy,zyl}.

Recently, starting from the research of Himonas et al. \cite{hk,hkm}, the continuity properties of the data-to-solution maps of the Camassa-Holm type equations are gradually attracting interest of many authors, see \cite{lyz1,lwyz}. Most of the non-uniform constinuity results are established only on a bounded set near the origin. To overcome this limitation, Inci obtained a series of nowhere uniform continuity results including many Camassa-Holm type equations \cite{inc1,inc2}. And for the incompressible Euler equation, Bourgain and Li \cite{bl} showed that the data-to-solution map is nowhere-uniform continuity in $H^s(\R^d)$ with $s\geq 0$ by using an idea of localized Galilean boost, this method will inspire us in this article.

As part of the well-posedness theory, the continuity properties of the data-to-solution map is indeed very important. In fact, the non-uniform continuity of data-to-solution map suggests that the local well-posedness cannot be established by the contraction mappings principle since this would imply Lipschitz continuity for the solution map. On the other hand, in some critical spaces the continuity of the data-to-solution maps are first broken before the existence and uniqueness of the solution, which leads to ill-posedness \cite{lyz2}.   

Most previous work on constinuity has focused on the spacial one-dimensional Camassa-Holm type equations equations, for the multi-dimensional Euler-Poincar\'{e} equations \eqref{E-P}, the continuity problem has not been thoroughly investigated. Until recently, Li et al. \cite{ldz} shown that the corresponding solution to \eqref{E-P} is not uniformly constinuous dependence for that the initial data in $H^{s}(\R^d),s>1+\frac d 2$. Later, the non-uniformly constinuous result was extended to Besov space $B^s_{p,r}(\R^d),s>\max\{1+\frac d2,\frac32\}$ in \cite{ldl}. 

It is worth to mention that, the non-uniform constinuity results of (\ref{E-P}) are established only on a bounded set near the origin. In this paper, we will remove the boundedness restriction and prove that the data-to-solution map of the Euler-Poincar\'{e} equations \eqref{E-P} is not uniformly continuous on any open subset $U\subset B^s_{p,r}(\R^d),s>\max\{1+\frac d2,\frac32\}$. Technically, our proof based on a symmetric form of the equation (\ref{E-P}), and a translation method to construct perturbation data, this method was introduced by Bourgain and Li \cite{bl} to proof the nowhere uniform constinuity of the incompressible Euler equations.

For simplicity, we first transform Eq.\eqref{E-P} into a transport type system. According to Yan\cite{yy}, we can rewrite \eqref{E-P} to the following nonlocal form:
\begin{align}\label{lns}
\partial_tu+u\cdot \nabla u=Q(u,u)+R(u,u),
\end{align}
where
\begin{equation}\label{qr}
  \begin{cases}
&Q(u,v)=-(I-\De)^{-1}\mathrm{div}\Big(\nabla u\nabla v+\nabla u\nabla v^T-\nabla u^T\nabla v-\nabla u(\mathrm{div} v)+\frac12(\nabla u : \nabla v)\mathbf{I}\Big),
\\&R(u,v)=-(I-\De)^{-1} \big( u~\div v+\nabla u^T v\big).
\end{cases}
\end{equation}
 We now define a symmetric bilinear operator $\T$ by
\begin{align}
  \T(u,v)&:=\frac12\Big( Q(u,v)+Q(v,u)+R(u,v)+R(v,u)\Big)\nonumber
  \\ &=-(I-\De)^{-1}\mathrm{div} \big(M(\nabla u,\nabla v)\big)-(I-\De)^{-1}\big(N(u,\nabla u; v,\nabla v)\big),\label{sbo}
  \end{align}
here $M, N$ are bilinear functions of $(\nabla u,\nabla v)$ and $(u,\nabla u; v, \nabla v)$ respectively according to (\ref{qr}), they are symmetric on $u, v$. Then, the Euler-Poincaré equations becomes
    \begin{equation}\label{eps}\tag{E-P}
      \begin{cases}
        \partial_tu+u\cdot \nabla u= \T(u,u), \qquad &(t,x)\in \R^+\times \R^d,\\
      u(0,x)=u_0,\qquad &x\in \R^d.
      \end{cases}
      \end{equation}
We first recall the non-uniform continuity results established in \cite{ldl}.
\begin{theorem}[\bf{Non-uniform continuity on a bounded set}]\label{T1} Let $d\geq 2$ and $s>2+\max\big\{1+\frac d p,\frac 3 2\big\}$ with $1\leq p,r\leq \infty$. The data-to-solution map $S_t$ for Euler-Poincaré equations (\ref{eps}) is not uniformly continuous from any bounded subset $O_N=\{u_0\in B^s_{p,r}(\R^d):\|u_0\|_{B^s_{p,r}}\leq N\}$ into $\mathcal{C}([0,T];B^s_{p,r})$. More precisely, there exists two sequences of initial data $f_n + g_n ,~f_n$ such that
\begin{align*}
  \|f_n\|_{B^s_{p,r}}\lesssim 1\quad and \quad \lim_{n\to\infty}\|g_n\|_{B^s_{p,r}}=0,
\end{align*}
with the solutions $S_t(f_n + g_n),~S_t(f_n)$ satisfy 
\begin{align*}
 \liminf_{n\to\infty}\|S_t(f_n + g_n)-S_t(f_n)\|_{B^s_{p,r}}\geq c_0t,~\forall t\in[0,T_0],
\end{align*}
for some constant $c_0>0$ and small time $T_0$.
\end{theorem}
The main result of this paper is the following theorem.
\begin{theorem}[\bf{Nowhere uniform continuity}]\label{T}
  Assume that $d\geq 2$, and  
  \begin{align}\label{spr}
    s>2+\max\big\{1+\frac d p,\frac 3 2\big\}\quad and\quad (p,r)\in (1,\infty)\times [1,\infty). 
  \end{align}
Then the data-to-solution map $S_t$ for Euler-Poincaré equations for the Cauchy problem (\ref{eps})
\begin{align*}
  S_t: B^s_{p,r}(\R^d)\to \mathcal{C}([0,T];B^s_{p,r}),\quad u_0\mapsto S_t(u_0),
\end{align*}
is nowhere uniformly continuous from $B^s_{p,r}$ into $\mathcal{C}([0,T];B^s_{p,r})$. More precisely, for any $u_0\in B^s_{p,r}$ and $N>0$, there exists two sequences of functions $f_n(x), g_n(x)$ such that
  \begin{align*}
    \|f_n\|_{B^s_{p,r}}\lesssim 2^{-N}\quad and \quad \lim_{n\to\infty}\|g_n\|_{B^s_{p,r}}=0,
  \end{align*}
  the corresponding solutions $S_t(f_n + g_n),~S_t(f_n)$ satisfy 
  \begin{align*}
   \liminf_{n\to\infty}\|S_t(u_0+f_n + g_n)-S_t(u_0+f_n)\|_{B^s_{p,r}}\geq c_0t,~\forall t\in[0,T_0],
  \end{align*}
  for some constant $c_0>0$ and small time $T_0$.
  \end{theorem}     
\begin{remark}
  As a comparison with Theorem \ref{T1}, Theorem \ref{T} avoids endpoints $p=1$ and $p=\infty$, this is because we need to use the boundedness of Riez transform in $L^p(\R^d)$ when doing gradient estimate of $\T$ (see Lemma \ref{3t} blow), which is only available when $p\in(1,\infty)$. 
\end{remark}
\begin{remark}
 The non-uniform constinuity in Theorem \ref{T1} established only on a bounded set near the origin, in Theorem \ref{T} we have removed these restrictions and showed that for any $u_0$ and any neighbour $U(u_0)\subset B^s_{p,r}$, the data-to-solution map restrict on $U$ is not uniformly continuous. In this sense, Theorem \ref{T} improves the previous results in \cite{ldl}. 
\end{remark}

The remainder of this paper is organized as follows. In Section \ref{sec2}, we list some notations and recall basic results of the Littlewood-Paley theory. In Section \ref{sec3}, we present the proof of Theorem \ref{T} by establishing some technical lemmas and propositions.

\section{Littlewood-Paley analysis}\label{sec2}
We first present some facts about the Littlewood-Paley decomposition, the nonhomogeneous Besov spaces and their some useful properties (see \cite{bcd} for more details).

Let $\mathcal{B}:=\{\xi\in\R^d:|\xi|\leq 4/3\}$ and $\mathcal{C}:=\{\xi\in\R^d:3/4\leq|\xi|\leq 8/3\}.$
Choose a radial, non-negative, smooth function $\chi:\R^d\mapsto [0,1]$ such that it is supported in $\mathcal{B}$ and $\chi\equiv1$ for $|\xi|\leq3/4$. Setting $\varphi(\xi):=\chi(\xi/2)-\chi(\xi)$, then we deduce that $\varphi$ is supported in $\mathcal{C}$. Moreover,
\begin{eqnarray*}
\chi(\xi)+\sum_{j\geq0}\varphi(2^{-j}\xi)=1 \quad \mbox{ for any } \xi\in \R^d.
\end{eqnarray*}
We should emphasize that the fact $\varphi(\xi)\equiv 1$ for $4/3\leq |\xi|\leq 3/2$ will be used in the sequel.

For every $u\in \mathcal{S'}(\mathbb{R}^d)$, the inhomogeneous dyadic blocks ${\Delta}_j$ are defined as follows
\begin{equation*}
{\Delta_ju=}
\begin{cases}
0,   &if \quad j\leq-2;\\
\chi(D)u=\mathcal{F}^{-1}(\chi \mathcal{F}u),  &if \quad j=-1;\\
\varphi(2^{-j}D)u=\mathcal{F}^{-1}\g(\varphi(2^{-j}\cdot)\mathcal{F}u\g), &if  \quad j\geq0.
\end{cases}
\end{equation*}
In the inhomogeneous case, the following Littlewood-Paley decomposition makes sense
$$
u=\sum_{j\geq-1}{\Delta}_ju\quad \text{for any}\;u\in \mathcal{S'}(\mathbb{R}^d).
$$
\begin{definition}label{besov}
Let $s\in\mathbb{R}$ and $(p,r)\in[1, \infty]^2$. The nonhomogeneous Besov space $B^{s}_{p,r}(\R^d)$ is defined by
\begin{align*}
B^{s}_{p,r}(\R^d):=\Big\{f\in \mathcal{S}'(\R^d):\;\|f\|_{B^{s}_{p,r}(\mathbb{R}^d)}<\infty\Big\},
\end{align*}
where
\begin{numcases}{\|f\|_{B^{s}_{p,r}(\mathbb{R}^d)}=}
\left(\sum_{j\geq-1}2^{sjr}\|\Delta_jf\|^r_{L^p(\mathbb{R}^d)}\right)^{\fr1r}, &if $1\leq r<\infty$,\nonumber\\
\sup_{j\geq-1}2^{sj}\|\Delta_jf\|_{L^p(\mathbb{R}^d)}, &if $r=\infty$.\nonumber
\end{numcases}
\end{definition}
The following Bernstein's inequalities will be used in the sequel.
\begin{lemma} Let $\mathcal{B}$ be a Ball and $\mathcal{C}$ be an annulus. There exist constants $C>0$ such that for all $k\in \mathbb{N}\cup \{0\}$, any positive real number $\lambda$ and any function $f\in L^p(\R^d)$ with $1\leq p \leq q \leq \infty$, we have
\begin{align*}
&{\rm{supp}}\hat{f}\subset \lambda \mathcal{B}\;\Rightarrow\; \|D^kf\|_{L^q}:=\sup_{|\alpha|=k}\|\partial^\alpha f\|_{L^q}\leq C^{k+1}\lambda^{k+(\frac{d}{p}-\frac{d}{q})}\|f\|_{L^p},  \\
&{\rm{supp}}\hat{f}\subset \lambda \mathcal{C}\;\Rightarrow\; C^{-k-1}\lambda^k\|f\|_{L^p} \leq \|\D^kf\|_{L^p} \leq C^{k+1}\lambda^k\|f\|_{L^p}.
\end{align*}
\end{lemma}
\begin{lemma}[See \cite{bcd}]\label{inte}
Let $(s_1,s_2,p,r)\in \R^2\times [1,\infty]^2$, and $s_1<s_2,~0<\theta<1$, then we have
  \begin{align*}
    \|u\|_{B^{\theta s_1+(1-\theta)s_2}_{p,r}}\leq &\|u\|_{B^{s_1}_{p,r}}^\theta\|u\|_{B^{s_2}_{p,r}}^{1-\theta},\\
    \|u\|_{B^{\theta s_1+(1-\theta)s_2}_{p,1}}\leq &\fr{C}{s_2-s_1}\Bigl(\frac{1}{\theta}+\frac{1}{1-\theta}\Bigr)\|u\|_{B^{s_1}_{p,\infty}}^\theta\|u\|_{B^{s_2}_{p,\infty}}^{1-\theta}.
      \end{align*}
\end{lemma}
Then, we give some important product estimates which will be used throughout the paper.
\begin{lemma}[See \cite{bcd}]\label{pe}
  For $(p,r)\in[1, \infty]^2$ and $s>0$, $B^s_{p,r}(\R^d)\cap L^\infty(\R^d)$ is an algebra. Moreover, for any $u,v \in B^s_{p,r}(\R^d)\cap L^\infty(\R^d)$, we have
  \bbal
  &\|uv\|_{B^{s}_{p,r}}\leq C(\|u\|_{B^{s}_{p,r}}\|v\|_{L^\infty}+\|v\|_{B^{s}_{p,r}}\|u\|_{L^\infty}).
  \end{align*}
In addition, if $s>\max\big\{1+\frac d p,\frac32\big\}$, then
    \begin{align*}
    &\|uv\|_{B^{s-2}_{p,r}(\R^{d})}\leq C\|u\|_{B^{s-2}_{p,r}(\R^{d})}\|v\|_{B^{s-1}_{p,r}(\R^{d})}.
    \end{align*}
    \end{lemma}

\begin{lemma}[See \cite{bcd,liy}]\label{tee}
  Let $(p,r)\in[1, \infty]^2$ and $\sigma\geq-\min\big\{\frac d p, 1-\frac d p\big\}$. Assume that $f_0\in B^\sigma_{p,r}(\R^d)$, $g\in L^1([0,T]; B^\sigma_{p,r}(\R^d))$ and $\na\mathbf{u}\in
  L^1([0,T]; B^{\sigma-1}_{p,r}(\R^d))$ if $\sigma>1+\fr d p$  or $\sigma=1+\fr d p, r=1$.
  If $f\in L^\infty([0,T]; B^\sigma_{p,r}(\R^d))\cap \mathcal{C}([0,T]; \mathcal{S}'(\R^d))$ solves the following linear transport equation:
  \begin{equation*}
  \quad \partial_t f+\mathbf{u}\cdot\na f=g,\quad \; f|_{t=0} =f_0.
  \end{equation*}
1. There exists a constant $C=C(\sigma,p,r)$ such that the following statement holds
  \begin{align*}
  \|f(t)\|_{B^\sigma_{p,r}}\leq e^{CV(t)} \Big(\|f_0\|_{B^\sigma_{p,r}}+\int_0^t e^{-CV(\tau)} \|g(\tau)\|_{B^\sigma_{p,r}}\mathrm{d}\tau\Big),
  \end{align*}
  where
  $$V(t)=\int_0^t \|\na \mathbf{u}(\tau)\|_{B^{\sigma-1}_{p,r}}\mathrm{d}\tau\quad \text{if}\quad \sigma>1+{\fr d p} \quad \text{or} \quad \{\sigma=1+\fr d p,\; r=1\}.$$
2. If $\sigma> 0$, then there exists a constant $C = C(\sigma,p,r )$ such that the following holds
\begin{align*}
  \|f(t)\|_{B^\sigma_{p,r}}\leq &\|f_0\|_{B^\sigma_{p,r}}+\int_0^t \|g(\tau)\|_{B^\sigma_{p,r}}\mathrm{d}\tau\\
  &+\int_0^t\Big( \|f(\tau)\|_{B^\sigma_{p,r}}\|\na \mathbf{u}\|_{L^\infty}+\|\na \mathbf{u}\|_{B^{\sigma-1}_{p,r}}\|\na f(\tau)\|_{L^\infty}\Big)\mathrm{d}\tau.
  \end{align*}
  \end{lemma}

\section{Proof of the main theorem }\label{sec3}

We first recall the local existence and uniqueness theory of solutions for the Cauchy problem \eqref{E-P} in Besov spaces \cite{yy}, then provide some technical lemmas and propositions.
\subsection{Preparation and technical lemmas}
\begin{lemma}[See \cite{yy}]\label{hlocal}
Assume that 
\begin{align}\label{dprs}
  d\in \mathbb N_+, 1\leq p,r\leq\infty~ and ~s>\max\{1+\frac{d}{p},\frac{3}{2}\}. 
  \end{align}
  Let $u_{0}\in B^{s}_{p,r}(\mathbb{R}^{d})$, then there exists a time $T=T(\|u_{0}\|_{B^{s}_{p,r}(\mathbb{R}^{d})})>0$ such that \eqref{E-P} has a unique solution in
\begin{equation*}
  \begin{cases}
  C([0,T];B^{s}_{p,r}(\R^d))\cap C^1([0,T];B^{s-1}_{p,r}(\R^d)), &if~~ r<\infty,\\
  L^\infty([0,T];B^{s}_{p,\infty}(\R^d))\cap Lip([0,T];B^{s-1}_{p,\infty}(\R^d)), &if~~ r=\infty.
  \end{cases}
\end{equation*}
 And the mapping $u_0\mapsto u$ is continuous from $B^{s}_{p,r}(\R^d)$ into 
$C([0,T];B^{s'}_{p,r}(\R^d))\cap C^1([0,T];B^{s'-1}_{p,r}(\R^d))$
for all $s'<s$ if $r=\infty$, and $s'=s$ otherwise. Moreover, for all $t\in[0,T]$, there holds
\begin{align*}
\|u(t)\|_{B^{s}_{p,r}(\mathbb{R}^{d})}\leq C\|u_{0}\|_{B^{s}_{p,r}(\mathbb{R}^{d})}.
\end{align*}
\end{lemma}

\begin{lemma}\label{3t}
  Let $(s,p,r)$ satisfy (\ref{spr}), then for the symmetric bilinear operator  $\T(f,g)$ defined by (\ref{qr}) and (\ref{sbo}), we have
       \begin{align}\label{etfg}
       &\|\T(f,g)\|_{B^{s}_{p,r}}\leq C\|f\|_{B^{s}_{p,r}}\|g\|_{B^{s}_{p,r}}
       \end{align}
   If $0<p<\infty$, there holds
   \begin{align}
     \|\T (f,g)\|_{L^p} &\leq \|\nabla f\|_{L^p}\|g, \nabla g\|_{L^\infty} \label{et1}\\
     \|\T (f,g)\|_{L^p} &\leq \sum_{0\leq |a|,|b|\leq 1}\|\partial^af\partial^bg\|_{L^p}=W_{1,p}(f,g)\label{et2}
   \end{align}
   And, for the gradient $\na\T$, we have
   \begin{align}
     \|\nabla\T (f,g)\|_{L^p} &\leq \|\nabla f\|_{L^p}\|g, \nabla g\|_{L^\infty}\label{egt1}\\
     \|\nabla\T (f,g)\|_{L^p}  &\leq \sum_{0\leq |a|,|b|\leq 2}\|\partial^af\partial^bg\|_{L^p}=W_{2,p}(f,g)\label{egt2}
   \end{align}
   Where we denote $\displaystyle{ W_{m,p}(f,g)=\sum_{0\leq |a|,|b|\leq m}\|\partial^af\partial^bg\|_{L^p}}$ with the multiindex $a=(a_1,a_2,\cdots,a_d),~|a|=a_1+\cdots+a_d$ and $\pa^a=\fr{\pa^{|a|}}{\pa x_1^{a_1}\cdots \pa x_d^{a_d}}$.
   \end{lemma}
   \begin{proof}
   As the operator $(I-\Delta)^{-1}$ is a Fourier $S^{-2}$-multiplier, it's easy to see that
   \begin{align*}
     \|\T(f,g)\|_{B^{s}_{p,r}}\leq C\|M(\nabla f,\nabla g)\|_{B^{s-1}_{p,r}}+C\|N(f,\nabla f, g,\nabla g)\|_{B^{s-2}_{p,r}}\leq C\|f\|_{B^{s}_{p,r}}\|g\|_{B^{s}_{p,r}},
     \end{align*}
   here we have use the Lemma \ref{pe}. Then in $L^p$ spaces,
     \begin{align*}
     \|\T (f,g)\|_{L^p} &= \|(I-\De)^{-1}\mathrm{div} \big(M(\nabla f,\nabla g)\big)+(I-\De)^{-1}\big(N(f,\nabla f, g,\nabla g)\big)\|_{L^p}
     \\ &\leq \|M(\nabla f,\nabla g)\|_{L^p}+\|N(f,\nabla f, g,\nabla g)\|_{L^p}
     \\ &\leq \|\nabla f\|_{L^p}\|\nabla g\|_{L^\infty}+\|\nabla f\|_{L^p}\| g\|_{L^\infty}
     \\ &\leq \|\nabla f\|_{L^p}\|g, \nabla g\|_{L^\infty}
     \end{align*}
     we also have
     \begin{align*}
       \|\T (f,g)\|_{L^p} &\leq \|M(\nabla f,\nabla g)\|_{L^p}+\|N(f,\nabla f, g,\nabla g)\|_{L^p}
       \\ &\leq \sum_{0\leq |a|,|b|\leq 1}\|\partial^af\partial^bg\|_{L^p}=W_{1,p}(f,g)
     \end{align*}  
     For the gradient $\na\T$, noting that $(I-\De)^{-1}\pa_i\pa_j=-\De(I-\De)^{-1}\big((-\De)^{-1}\pa_i\pa_j\big)=\big((1-\De)^{-1}+1\big)R_iR_j$ and the Riesz transform $R_i$ is bounded in $L^p\to L^p,~p\in(1,\infty)$, then we have
     \begin{align*}
     \|\nabla\T (f,g)\|_{L^p} &= \|\nabla(I-\De)^{-1}\mathrm{div} \big(M(\nabla f,\nabla g)\big)+\nabla(I-\De)^{-1}\big(N(f,\nabla f, g,\nabla g)\big)\|_{L^p}
     \\ &\leq \|M(\nabla f,\nabla g)\|_{L^p}+\|N(f,\nabla f, g,\nabla g)\|_{L^p}
     \\ &\leq \|\nabla f\|_{L^p}\|\nabla g\|_{L^\infty}+\|\nabla f\|_{L^p}\| g\|_{L^\infty}
     \\ &\leq \|\nabla f\|_{L^p}\|g, \nabla g\|_{L^\infty}
     \end{align*}
     and
     \begin{align*}
      \|\nabla\T (f,g)\|_{L^p} &\leq \|\div M(\nabla f,\nabla g)\|_{L^p}+\|N(f,\nabla f, g,\nabla g)\|_{L^p}
      \\ &\leq \sum_{0\leq |a|,|b|\leq 2}\|\partial^af\partial^bg\|_{L^p}=W_{2,p}(f,g)
     \end{align*}
   \end{proof}

We'll need the following estimates of the difference $u(t)-v(t)$ in Besov spaces.

\begin{proposition}\label{df1}
  Let $ 1\leq p,r\leq\infty$ and $s>\max\{1+\frac{d}{p},\frac{3}{2}\}$. Assume that  $u(t),v(t)$ are solutions of (\ref{eps}) with initial data $(u_{0},v_0)\in B^{s}_{p,r}(\mathbb{R}^{d})$, then $\delta(t):=u(t)-v(t)$ satisfies
  \begin{align*}
\|\delta(t)\|_{B^{s-1}_{p,r}}\leq \|\delta_0\|_{B^{s-1}_{p,r}}\exp\big(C\int_0^t\|u(\tau),v(\tau)\|_{B^{s}_{p,r}}d\tau\big)
  \end{align*}
  and
  \begin{align}
    \|\delta(t)\|_{B^{s}_{p,r}}\leq \Big(\|\delta_0\|_{B^{s}_{p,r}}+C\int_0^t\|\delta\|_{B^{s-1}_{p,r}}\|\nabla v\|_{B^{s}_{p,r}}d\tau\Big)\exp\big(C\int_0^t\|u(\tau),v(\tau)\|_{B^{s}_{p,r}}d\tau\big)\label{del2}
      \end{align}
    \end{proposition}
\begin{proof}
The first inequality has been proved in \cite{yy}, it remains to prove (\ref{del2}). As $\T$ is a symmetric bilinear operator, it's easy to deduce that $\delta=u-v$ solves the transport equation
\begin{align}\label{del}
  \pa_t\delta+u\cdot\na \delta=-\delta\cdot\na v+\T(\delta,u+v).
\end{align}
Then, by Lemma \ref{tee} and \ref{3t}
\begin{align*}
  \|\delta(t)\|_{B^{s}_{p,r}}\leq& \|\delta_0\|_{B^{s}_{p,r}}+C\int_0^t\Big(\|u\|_{B^{s}_{p,r}}\|\delta\|_{B^{s}_{p,r}}+\|\delta\cdot\nabla v\|_{B^{s}_{p,r}} +\|\T(\delta,u+v)\|_{B^{s}_{p,r}}\Big)d\tau \\
  \leq & \|\delta_0\|_{B^{s}_{p,r}}+C\int_0^t\Big(\|u(\tau),v(\tau)\|_{B^{s}_{p,r}}\|\delta\|_{B^{s}_{p,r}}+\|\delta\|_{B^{s-1}_{p,r}}\|\nabla v\|_{B^{s}_{p,r}}\Big)d\tau.
\end{align*}
now (\ref{del2}) is direct result from Gronwall's inequality.
\end{proof}
\begin{proposition}\label{df2}
  Suppose $\su(t),u(t),v(t)$ are the solutions of (\ref{eps}) of initial data $u_0+v_0, u_0, v_0$ respectively. Then, under the assumptions of (\ref{spr}), we have
\begin{align*}
\|\su-u-v\|_{B^{s}_{p,r}}\leq C \|u_0,v_0\|_{B^{s+1}_{p,\infty}}^{1-\theta}\exp\Big(\|u_0,v_0\|_{B^{s+1}_{p,r}}\theta\Big)\Big(\int_{0}^{t}\sL (u,v) d\tau\Big)^\theta,
\end{align*}
where $\theta=\frac{1}{s+1}$ and use the notation $\displaystyle{W_{2,p}(u,v)=\sum_{0\leq |a|,|b|\leq 2}\|\partial^au\partial^bv\|_{L^p}}$.
\end{proposition}
\begin{proof}
Since $\su(t),u(t),v(t)$ are solutions of 
\begin{equation*}
  \begin{cases}
    \partial_t\su+\su\cdot \nabla \su= \T(\su,\su), \qquad &\su(0)=u_0+v_0,\\
    \partial_tu+u\cdot \nabla u= \T(u,u), \qquad &u(0)=u_0,\\
    \partial_tv+v\cdot \nabla v= \T(v,v), \qquad &v(0)=v_0.
  \end{cases}
  \end{equation*}
by the symmetry and linearity of $\T$, we can deduce that $w(t)=\su(t)-u(t)-v(t)$ satisfies
\begin{equation}\label{weq}
  \begin{cases}
    \partial_tw+\su\cdot \nabla w=&-w\cdot \nabla(u+v)+ \T(w,\su+u+v)\\
    &-u\cdot \nabla v-v\cdot \nabla u-2\T(u,v),\\
    \qquad \qquad w(0)=&0.
  \end{cases}
  \end{equation}
  By the interpolation inequality (see Lemma \ref{inte} ), we obtain
  \begin{align}\label{iw}
\|w\|_{B^{s}_{p,r}}\leq C\|w\|_{B^0_{p,\infty}}^\theta\|w\|_{B^{s+1}_{p,\infty}}^{1-\theta}\leq \|u_0,v_0\|_{B^{s+1}_{p,\infty}}^{1-\theta}\|w\|_{L^p}^\theta
  \end{align}
The rest of the proof is to bound the $L^p$ norm of $w$, taking the inner product of  (\ref{weq}) with $\w ^{p-1}:=(|w_1|^{p-2}w_1,|w_2|^{p-2}w_2,\cdots,|w_d|^{p-2}w_d)$,  we obtain
\begin{align}\label{ew1}
\frac1p\frac{d}{dt}\|w\|_{L^p}^p=&\sum_{i=1}^d\int p^{-1}|w_i|^p(\div \su) dx-\sum_{i,j}\int \w_i^{p-1}w_j\partial_j(u_i+v_i) dx\nonumber
\\ &+\sum_{i=1}^d\int \w_i^{p-1}\T_i(w,\su+u+v) dx-\sum_{i=1}^d\int \w_i^{p-1}(u\cdot \nabla v_i+v\cdot \nabla u_i+\T_i(u,v) dx\nonumber
\\ \leq& \frac1p\|\div \su\|_{L^\infty}\|w\|_{L^p}^p+C_d(\|\nabla u\|_{L^\infty}+\|\nabla v\|_{L^\infty})\|w\|_{L^p}^p \nonumber
\\&+\|w\|_{L^p}^{p-1}\|\T(w,\su+u+v)\|_{L^p}+\|w\|_{L^p}^{p-1}\|u\cdot \nabla v+v\cdot \nabla u+\T(u,v)\|_{L^p}
\end{align}
Thanks to the estimates of $\T$ in Lemma \ref{3t}, in particular take (\ref{et1}), (\ref{et2}), into (\ref{ew1}) we have 
  \begin{align*}
    \frac{d}{dt}\|w\|_{L^p} &\leq C(\|\div \su\|_{L^\infty}+\|\nabla u\|_{L^\infty}+\|\nabla v\|_{L^\infty})\|w\|_{L^p}
    \\&+C(\|\su,\nabla \su\|_{L^\infty}+\|u,\nabla u\|_{L^\infty}+\|v,\nabla v\|_{L^\infty})\|\nabla w\|_{L^p}+\sL(u,v)
    \\ &\leq \|u_0,v_0\|_{B^{s}_{p,r}}(\|w\|_{L^p}+\|\nabla w\|_{L^p})+\sL(u,v)
    \end{align*}
Now, we should bound the gradient matrix $\nabla w$, take the gradient to (\ref{weq}), then in components
\begin{align*}
\pa_t\pa_jw_i=&-\su_k\pa_k\pa_jw_i-\pa_j\su_k\pa_kw_i+\pa_jT_i(w,\su+u+v)
\\ &-w_k\pa_k(\pa_ju_i+\pa_jv_i)-\pa_jw_k\pa_k(u_i+v_i)-\pa_j\big(u_k\pa_kv_i-v_k\pa_k u_i-2\T_i(u,v)\big)
\end{align*}
Taking the $L^2$ inner product with $\w_{i,j}^{p-1}:=|\pa_jw_i|^{p-2}\pa_jw_i$ and sum the indices $i,j$, we get
\begin{align}\label{egw1}
  \frac1p\frac{d}{dt}\|\nabla w\|_{L^p}^p=&\sum_{1\leq i,j\leq d}\int p^{-1}|\pa_jw_i|^p(\div \su) dx-\int \nabla\w^{p-1}:(\nabla w \nabla \su)dx+\int \nabla\w^{p-1}:\nabla\T(w,\su+u+v) dx\nonumber
  \\ &-\int \nabla\w^{p-1}:\big(w\cdot \nabla (\nabla u+\nabla v)\big) dx-\int \nabla\w^{p-1}:\big((\nabla u+\nabla v)\nabla w\big) dx\nonumber
  \\ &-\int \nabla\w^{p-1}:\nabla(u\cdot\nabla v+v\cdot\nabla u+2\T(u,v))dx\nonumber
  \\ \leq& \frac1p\|\div \su\|_{L^\infty}\|\nabla w\|_{L^p}^p+C_d\|\nabla \su\|_{L^\infty}\|\nabla w\|_{L^p}^p +\|\nabla w\|_{L^p}^{p-1}\|\nabla\T(w,\su+u+v) \|_{L^p}\nonumber
  \\&+\|\nabla w\|_{L^p}^{p-1}\|w\|_{L^p}\big(\|\nabla^2 u\|_{L^\infty}+\|\nabla^2 v\|_{L^\infty}\big)+\|\nabla w\|_{L^p}^p\big(\|\nabla u\|_{L^\infty}+\|\nabla v\|_{L^\infty}\big)\nonumber
  \\&+|\nabla w\|_{L^p}^{p-1}\|\nabla\big(u\cdot \nabla v+v\cdot \nabla u+2\T(u,v)\big)\|_{L^p}
\end{align}
where we denote $\nabla\w^{p-1}=(\w_{i,j}^{p-1})_{d\times d}$
and $A:B:=\sum_{i,j}a_{i,j}b_{i,j}$. Again using Proposition \ref{3t} for the matrix operator $\nabla\T$, by plug (\ref{egt1}),(\ref{egt2}) into (\ref{egw1}), we obtain
\begin{align*}
  \frac{d}{dt}\|\nabla w\|_{L^p} &\leq C(\|\nabla \su\|_{L^\infty}+\|\nabla u\|_{L^\infty}+\|\nabla v\|_{L^\infty})\|\nabla w\|_{L^p}+\|w\|_{L^p}\big(\|\nabla^2 u\|_{L^\infty}+\|\nabla^2 v\|_{L^\infty}\big)\nonumber
  \\&+\|\nabla\T(w,\su+u+v) \|_{L^p}+\|\nabla\big(u\cdot \nabla v+v\cdot \nabla u+2\T(u,v)\big)\|_{L^p}\nonumber
  \\&\leq C(\|\su,\nabla \su\|_{L^\infty}+\|u,\nabla u\|_{L^\infty}+\|v,\nabla v\|_{L^\infty})\|\nabla w\|_{L^p}\nonumber
  \\&+\big(\|\nabla^2 u\|_{L^\infty}+\|\nabla^2 v\|_{L^\infty}\big)\|w\|_{L^p}+\sL(u,v)
  \\ &\leq C\|u_0,v_0\|_{B^{s+1}_{p,r}}(\|w\|_{L^p}+\|\nabla w\|_{L^p})+\sL(u,v)
  \end{align*}
  Combining (3.14) and (3.17) yields that
  \begin{align*}
    \frac{d}{dt}\|w,\nabla w\|_{L^p} &\leq C\|u_0,v_0\|_{B^{s+1}_{p,r}}(\|w\|_{L^p}+\|\nabla w\|_{L^p})+2 \sL(u,v)
  \end{align*}
By Gronwall's inequality and (\ref{iw}) we complete the proof.
\end{proof}
\begin{remark}
  The proofs of Proposition \ref{df1} and \ref{df2} rely on the symmetry of $\T$, especially when it comes to getting simplified equations (\ref{del}) and (\ref{weq}). Most previous studies on the well-posedness of Euler-Poincaré equations use the bilinear form (\ref{lns}), the lack of symmetry makes the calculation complicated. Infact, when $d=1$ namely the Camassa-Holm equation has the transport form $\pa_tu+u\pa_xu=P(u,u)$ with $P(u,v)=-\pa_x(1-\pa_x^2)^{-1}\bigl(uv+\fr12(\pa_xu\pa_xv)\bigr)$ is symmetric by default. In this respect, our new form (\ref{eps}) is a more natural high-dimensional generalization of the CH equation. 
\end{remark}
\subsection{Construction of Perturbation Data}
For localization in the Fourier domain, we introduce the following bump function in the frequency space.
Let $\widehat{\phi}\in \mathcal{C}^\infty_0(\mathbb{R})$ be a non-negative and even function satisfy
\begin{numcases}{\widehat{\phi}(\xi)=}
1,&if $|\xi|\leq \frac{1}{4}$,\nonumber\\
0,&if $|\xi|\geq \frac{1}{2}$.\nonumber
\end{numcases}
and let
\begin{equation}\label{fg}
  \begin{cases}
    &f_n=2^{-ns-N}\big(\cos(\frac{17}{12}2^{n}x_1)\phi(x_1)\phi(x_2)\cdots\phi(x_d),0,\cdots,0\big)
    \\ &g_n=\big(2^{-n}\phi(x_1)\phi(x_2)\cdots\phi(x_d),0,\cdots,0\big).
    \end{cases}
  \end{equation} 
We define the perturbation data by adding a translation transform
    \begin{equation}\label{fgm}
      \begin{cases}
&f_n^m=f_n(x_1-m,x_2,\cdots,x_d)
\\&g_n^m=g_n(x_1-m,x_2,\cdots,x_d)
\end{cases}
\end{equation}
 Noting that $\widehat{f_n^m}$ is supported in $[-\frac12,\frac12]^d\pm(\frac{17}{12}2^{n},0,\cdots,0),$ this support set is completely covered by the ring $C_n= \{\xi\in\mathbb{R}^d:\frac432^{n}\leq|\xi|\leq \frac322^{n}\}.$ Thus, by the definition of $\Delta_j,$ we know
\begin{numcases}{\Delta_j(f_n)=}
f_n^m, &if $j=n$,\nonumber\\
0, &if $j\neq n$.
\end{numcases}
On account of above and the definition of Besov space, we can show that for $k\in\R$ 
\begin{align}\label{bfn}
\|f_n^m\|_{B^{s+k}_{p,r}}\leq C2^{kn-N} \qquad and\qquad \|g_n^m\|_{B^{s+k}_{p,r}}\to 0\qquad for~ n\rightarrow  \infty
\end{align}
By the previous work \cite{ldl} and translation invariance of the system (\ref{eps}), we know that, for the corresponding solutions $S_t(f_n^m+g_n^m)$ and $S_t(f_n^m)$ there is a positive constant $c_0$ and a small time $T_0$, such that for any $t\in[0,T_0]$,
\begin{align}\label{sfg}
\liminf_{n\to\infty} \|S_t(f_n^m+g_n^m)-S_t(f_n^m)\|_{B^{s}_{p,r}}\geq c_0t.
\end{align}
\subsection{Proof of Theorem \ref{T}}
Roughly speaking, our proof of Theorem \ref{T} based on the following approximation
\begin{align}
 \qquad \qquad S_t(u_0+f_n^m+g_n^m)&-S_t(u_0+f_n^m)\tag{I}\label{A}
  \\ = S_t(S_nu_0+f_n^m+g_n^m)&-S_t(S_nu_0+f_n^m)+\mathcal{E}_n^m \tag{II}\label{B}
  \\ = \Big(S_t(S_nu_0)+S_t(f_n^m+g_n^m)\Big)&-\Big(S_t(S_nu_0)+S_t(f_n^m)\Big)+\mathcal{E}_{n,m} \nonumber
  \\=S_t(f_n^m+g_n^m)&-S_t(f_n^m)+\mathcal{E}_{n,m}, \tag{III}\label{C}
\end{align}
with some small error terms $\mathcal{E}_n^m,~\mathcal{E}_{n,m}$. More precisely, we devide (\ref{A}) into three parts
\begin{align}\label{app}
  \begin{split}
 &S_t(u_0+f_n^m+g_n^m)-S_t(u_0+f_n^m)=\\
 &\underbrace{\Big(S_t(u_0+f_n^m+g_n^m)-S_t(S_nu_0+f_n^m+g_n^m)\Big)-\Big(S_t(u_0+f_n^m)-S_t(S_nu_0+f_n^m)\Big)}_{\mathcal{E}_{n}^m}+\\
 &\underbrace{\Big(S_t(S_nu_0+f_n^m+g_n^m)-S_t(S_nu_0)-S_t(f_n^m+g_n^m)\Big)-\Big(S_t(S_nu_0+f_n^m)-S_t(S_nu_0)-S_t(f_n^m)\Big)}_{\mathcal{E}_{n,m}}\\ 
&+ S_t(f_n^m+g_n^m)-S_t(f_n^m).
\end{split}
 \end{align}
 We proof the approximation $(\ref{C})\to(\ref{B})\to(\ref{A})$ in the following sense.
 \begin{proposition}\label{err12}
  Let $f_n^m,~g_n^m$ be the perturbation data defined by (\ref{fg}) and (\ref{fgm}), then for any initial data $u_0\in B^{s}_{p,r}$ with $\|u_0\|_{B^{s}_{p,r}}=\ro$, the error terms $\mathcal{E}_n^m,~\mathcal{E}_{n,m}$ in (\ref{app}) satisfy 
 \begin{align}
  &\sup_{m,t}\|\mathcal{E}_n^m\|_{B^{s}_{p,r}}\leq C_\ro\|(I-S_n)u_0\|_{B^{s}_{p,r}},\label{err1}\\ &\lim_{m\to\infty}\bigl(\sup_{0\leq t\leq T}\|\mathcal{E}_{n,m}\|_{B^{s}_{p,r}}\bigr)=0\quad for~any~fixed~n.\label{err2}
   \end{align}.
  \end{proposition}
\begin{proof}
We first to handle (\ref{err1}). Using proposition \ref{df1} with $\delta(t)=S_t(u_0+f_n^m)-S_t(S_nu_0+f_n^m)$, as $\|u_0+f_n^m\|_{B^s_{p,r}}\approx \|S_nu_0+f_n^m\|_{B^s_{p,r}}$ for $ m\in\R$ and $n\gg 1$, the solution sequences have a common lifespan $T\approx T^*(\|u_0\|_{B^s_{p,r}})$,  then for any $t\in[0,T)$ we have
\begin{align}\label{dif1}
  \|\delta(t)\|_{B^s_{p,r}} \leq &\Big(\|(I-S_n)u_0\|_{B^s_{p,r}}+\int_0^t\|\delta(\tau)\|_{B^{s-1}_{p,r}}\|\nabla S_t(S_nu_0+f_n^m)\|_{B^s_{p,r}}d\tau\Big)\nonumber\\
  &\cdot\exp\big(\int_0^t\|S_t(u_0+f_n^m),S_t(S_nu_0+f_n^m)\|_{B^{s}_{p,r}}d\tau\big)\nonumber\\
  \leq &\Big(\|(I-S_n)u_0\|_{B^s_{p,r}}+\int_0^t\|\delta(\tau)\|_{B^{s-1}_{p,r}}\|S_nu_0+f_n^m\|_{B^{s+1}_{p,r}}d\tau\Big)\nonumber\\
  &\cdot\exp\big(\int_0^t\|u_0+f_n^m,S_nu_0+f_n^m\|_{B^{s}_{p,r}}d\tau\big)\nonumber\\
  \leq &C_{\ro}\Big(\|(I-S_n)u_0\|_{B^s_{p,r}}+\int_0^t\|\delta(\tau)\|_{B^{s-1}_{p,r}}\cdot 2^nd\tau\Big)
\end{align}
and
\begin{align}\label{dif2}
  \|\delta(t)\|_{B^{s-1}_{p,r}} \leq &\|(I-S_n)u_0\|_{B^{s-1}_{p,r}}\exp\big(\int_0^t\|S_t(u_0+f_n^m),S_t(S_nu_0+f_n^m)\|_{B^{s}_{p,r}}d\tau\big)\nonumber\\
  \leq &C_{\ro}2^{-n}\|(I-S_n)u_0\|_{B^s_{p,r}}
\end{align}
take (\ref{dif2}) into (\ref{dif1}) we get 
\begin{align}\label{dif3}
  \|\delta(t)\|_{B^{s}_{p,r}} \leq C_\ro\|(I-S_n)u_0\|_{B^{s}_{p,r}}.
\end{align}
As in (\ref{dif3}) the $C_\ro$ not depend on the translation parameter $m$ and $t\in[0,T]$, then we have
\begin{align}\label{dif0}
 \sup_{m,t}\|S_t(u_0+f_n^m)-S_t(S_nu_0+f_n^m)\|_{B^{s}_{p,r}}=\sup_{m,t}\|\delta(t)\|_{B^{s}_{p,r}}\leq C_\ro\|(I-S_n)u_0\|_{B^{s}_{p,r}}.
\end{align}
 With exactly the same argument, we can deduce that
 $$\sup_{m,t}\|S_t(u_0+f_n^m+g_n^m)-S_t(S_nu_0+f_n^m+g_n^m)\|_{B^{s}_{p,r}}\leq C_\ro\|(I-S_n)u_0\|_{B^{s}_{p,r}},$$ 
 along with (\ref{dif0}), we complete the proof of (\ref{err1}).

 In order to deduce (\ref{err2}), we should use Proposition \ref{df2} with the setting 
 $\su(t)=S_t(S_nu_0+f_n^m), u(t)=S_t(S_nu_0)$ and $v(t)=S_t(f_n^m),$ and denote 
 $$w=\su-u-v=S_t(S_nu_0+f_n^m)-S_t(S_nu_0)-S_t(f_n^m).$$
 Since $u_0\in B^{s}_{p,r}$ and $\|f_n^m\|_{B^{s}_{p,r}}\approx 1$, it's easy to see that
 \begin{align}
  \|S_nu_0, f_n^m\|_{B^{s+1}_{p,r}}\leq C_\ro2^n,
 \end{align}
with $C_\ro$ only depend on $\ro:=\|u_0\|_{B^{s}_{p,r}}$, Then from Proposition \ref{df2} we know that
\begin{align}\label{diff1}
  \|w(t)\|_{B^{s}_{p,r}}\leq C_\ro 2^ne^{C_\ro 2^n\theta}\Bigl(\sum_{0\leq |a|,|b|\leq 2}\int_0^t\|\partial^aS_t(S_nu_0)\partial^bS_t(f_n^m)\|_{L^p}d\tau\Bigr)^\theta.
\end{align}
Notice that, by definition $\partial^bS_t(f_n^m)=\partial^bS_t(f_n(x_1-m,\cdots,x_d))=\partial^bS_t(f_n)(x_1-m,\cdots,x_d)$, for fixed $n$ and any $(t, x)$, considering that $S_t(f_n)$ is a smooth function decay at infinity, we have
\begin{align*}
  &\lim_{m\to\infty}\partial^aS_t(S_nu_0)(x)\partial^bS_t(f_n)(x_1-m,\cdots,x_d)=0,\\
  &|\partial^aS_t(S_nu_0)(x)\partial^bS_t(f_n)(x_1-m,\cdots,x_d)|\leq M|\partial^aS_t(S_nu_0)|(\tau,x)\in L^1\bigl([0,T],L^p(\R)\bigr). 
\end{align*}
By the Lebesgue Dominated Convergence Theorem, we have
\begin{align}\label{diff2}
  \lim_{m\to\infty}\int_0^T\|\partial^aS_t(S_nu_0)\partial^bS_t(f_n^m)\|_{L^p}d\tau=0
\end{align}
Then from (\ref{diff1}),(\ref{diff2}) we know that, for the fixed $n$ and any $t\in[0,T]$
\begin{align}\label{diff3}
  \lim_{m\to\infty}\sup_{0\leq t\leq T}\|w(t)\|_{B^{s}_{p,r}}=\lim_{m\to\infty}\sup_{0\leq t\leq T}\|S_t(S_nu_0+f_n^m)-S_t(S_nu_0)-S_t(f_n^m)\|_{B^{s}_{p,r}}=0,
\end{align}
with the same argument, we can also get that, for any fixed $n$ and $t\in[0,T]$
\begin{align}\label{diff4}
  \lim_{m\to\infty}\sup_{0\leq t\leq T}\|S_t(S_nu_0+f_n^m+g_n^m)-S_t(S_nu_0)-S_t(f_n^m+g_n^m)\|_{B^{s}_{p,r}}=0.
\end{align}
Combining (\ref{diff3}) and (\ref{diff4}), this yields (\ref{err2}).
\end{proof}
With (\ref{sfg}), (\ref{app}) and Propositions \ref{err12} in hand, we can complete our proof of Theorem \ref{T}. First of all, from the identity (\ref{app}) we know that for any time $t\in[0,T]$
\begin{align}\label{app3}
  \begin{split}
  &\quad\|S_t(u_0+f_n^m+g_n^m)-S_t(u_0+f_n^m)\|_{B^{s}_{p,r}}\\
  &\geq \|S_t(f_n^m+g_n^m)-S_t(f_n^m)\|_{B^{s}_{p,r}}
  -\sup_{m,t}\|\mathcal{E}_{n}^m\|_{B^{s}_{p,r}}-\sup_{0\leq t\leq T}\|\mathcal{E}_{n,m}\|_{B^{s}_{p,r}}
\end{split}
\end{align}
Then, by (\ref{err2}) in Proposition \ref{err12}, for the fixed $n$, we can find a sufficiently large $m_n$ such that
\begin{align*}
  \sup_{0\leq t\leq T}\|\mathcal{E}_{n,m_n}\|_{B^{s}_{p,r}}\leq 2^{-n}.
\end{align*}
 combining this and (\ref{err1}) in Proposition \ref{err12}, by (\ref{app3}) we get
\begin{align}\label{inf}
  &\quad\|S_t(u_0+f_n^{m_n}+g_n^{m_n})-S_t(u_0+f_n^{m_n})\|_{B^{s}_{p,r}}\nonumber\\
  &\geq \|S_t(f_n^{m_n}+g_n^{m_n})-S_t(f_n^{m_n})\|_{B^{s}_{p,r}}-C_\ro\|(I-S_n)u_0\|_{B^{s}_{p,r}}-2^{-n}.
\end{align}
As $u_0\in B^{s}_{p,r}$, that means $\|(I-S_n)u_0\|_{B^{s}_{p,r}}\to 0$ when $n\to\infty$. For the small $t\in[0,T_0]$, we already have (\ref{sfg}), it follows from (\ref{inf}) that
\begin{align}
  \liminf_{n\to\infty}\|S_t(u_0+f_n^{m_n}+g_n^{m_n})-S_t(u_0+f_n^{m_n})\|_{B^{s}_{p,r}}\geq c_0t,\quad \forall t\in[0,T_0].
\end{align}
And the sequences of initial data satisfy
\begin{align}
  \lim_{n\to\infty}\|(u_0+f_n^{m_n}+g_n^{m_n})-(u_0+f_n^{m_n})\|_{B^{s}_{p,r}}=\lim_{n\to\infty}\|g_n^{m_n}\|_{B^{s}_{p,r}}=0.
\end{align}
This complete the proof Theorem \ref{T}.

\vspace*{1em}
\noindent\textbf{Acknowledgements.}
M. Li was supported by Educational Commission Science Programm of Jiangxi Province
(No. GJJ190284) and Natural Science Foundation of Jiangxi Province (No. 20212BAB211011 and 20212BAB201008). 


\end{document}